\documentclass[11pt, notitlepage]{article}   

\usepackage{amsmath,amsthm,amsfonts,hyperref}   

\usepackage{amscd}
\usepackage{amsfonts}
\usepackage{amssymb}
\usepackage{color}      
\usepackage{epsfig}
\usepackage{graphicx}           

\usepackage{tikz}

\theoremstyle{plain}
\newtheorem{theorem}{Theorem}[section]
\newtheorem{lemma}[theorem]{Lemma}
\newtheorem{corollary}[theorem]{Corollary}
\newtheorem{proposition}[theorem]{Proposition}
\newtheorem{observation}[theorem]{Observation}
\newtheorem{remark}[theorem]{Remark}
\newtheorem{question}[theorem]{Question}

\theoremstyle{definition}

\newcommand{\cdim}{\textnormal{cdim}}
\newcommand{\diam}{\textnormal{diam}}

\def\finf{\mathop{{\rm I}\kern -.27 em {\rm F}}\nolimits}

\newcommand{\adim}{\textnormal{adim}}

\begin{document}


\title{The distance-$k$ dimension of graphs}

\author{{\bf{Jesse Geneson}}$^1$ and {\bf{Eunjeong Yi}}$^2$\\
\small San Jose State University, San Jose, CA 95192, USA$^1$\\
\small Texas A\&M University at Galveston, Galveston, TX 77553, USA$^{2}$\\
{\small\em jesse.geneson@sjsu.edu}$^1$; {\small\em yie@tamug.edu}$^2$}

\maketitle

\date{}

\begin{abstract}
The metric dimension, $\dim(G)$, of a graph $G$ is a graph parameter motivated by robot navigation that has been studied extensively. Let $G$ be a graph with vertex set $V(G)$, and let $d(x,y)$ denote the length of a shortest $x-y$ path in $G$. For a positive integer $k$ and for distinct $x,y \in V(G)$, let $d_k(x,y)=\min\{d(x,y), k+1\}$ and let $R_k\{x,y\}=\{z\in V(G): d_k(x,z) \neq d_k(y,z)\}$. A subset $S\subseteq V(G)$ is a \emph{distance-$k$ resolving set} of $G$ if $|S \cap R_k\{x,y\}| \ge 1$ for any pair of distinct $x,y \in V(G)$, and the \emph{distance-$k$ dimension}, $\dim_k(G)$, of $G$ is the minimum cardinality over all distance-$k$ resolving sets of $G$. In this paper, we study the distance-$k$ dimension of graphs. We obtain some general bounds for distance-$k$ dimension. For all $k \ge 1$, we characterize connected graphs $G$ of order $n$ with $\dim_k(G) \ge n-2$. We determine $\dim_k(G)$ when $G$ is a cycle or a path. We also examine the effect of vertex or edge deletion on the distance-$k$ dimension of graphs.
\end{abstract}



\section{Introduction}

Let $G$ be a finite, simple, undirected, and connected graph with vertex set $V(G)$ and edge set $E(G)$. The \emph{distance} between two vertices $x, y \in V(G)$, denoted by $d(x, y)$, is the length of a shortest path between $x$ and $y$ in $G$. Metric dimension, introduced by Slater~\cite{slater} and by Harary and Melter~\cite{harary}, is a graph parameter that has been studied extensively. For distinct $x,y \in V(G)$, let $R\{x,y\}=\{z\in V(G): d(x,z) \neq d(y,z)\}$. A subset $S\subseteq V(G)$ is a \emph{resolving set} of $G$ if $|S \cap R\{x,y\}| \ge 1$ for any pair of distinct vertices $x$ and $y$ in $G$. The \emph{metric dimension} of $G$, denoted by $\dim(G)$, is the minimum cardinality over all resolving sets of $G$. It is $NP$-hard in general to compute $\dim(G)$~\cite{tree2, NP}.

Khuller et al.~\cite{tree2} considered robot navigation as one of the applications of metric dimension, where a robot that moves from node to node knows its distances to a set of landmarks, which are placed on the elements of the resolving set. Assuming that a sensor that can detect long distance to landmarks can be costly, the authors of~\cite{adim} consider the situation where a robot can only detect landmarks that are adjacent to it. They define the adjacency dimension, $\adim(G)$, of $G$ to be the minimum number of such landmarks that are needed for the robot to determine its position. More generally, if the landmark detection range of a robot is $k>0$, then the minimum number of such landmarks needed to determine the robot's position on the graph is called the distance-$k$ dimension (see~\cite{broadcast}). 

For a positive integer $k$ and for $x,y \in V(G)$, let $d_k(x,y)=\min\{d(x,y),k+1\}$. For a positive integer $k$ and for distinct $x,y \in V(G)$, let $R_k\{x,y\}=\{z\in V(G): d_k(x,z) \neq d_k(y,z)\}$. A subset $S\subseteq V(G)$ is a \emph{distance-$k$ resolving set} of $G$ if $|S \cap R_k\{x,y\}| \ge 1$ for any pair of distinct vertices $x$ and $y$ in $G$, and the \emph{distance-$k$ dimension} of $G$, denoted by $\dim_k(G)$, is the minimum cardinality over all distance-$k$ resolving sets of $G$. The distance-$k$ dimension of graphs was studied in \cite{br}, where it was also investigated more generally for metric spaces. The complexity of the problem was studied in \cite{Juan} and \cite{eyr}, where it was shown that computing $\dim_k(G)$ is an NP-hard problem for any positive integer $k$. The graphs $G$ with $\dim_k(G) = 1$ were characterized in \cite{aem}, which also investigated the problem in a more general setting.

For an ordered set $S=\{u_1, u_2, \ldots, u_{\beta}\} \subseteq V(G)$ of distinct vertices, the metric code and the distance-$k$ metric code, respectively, of $v \in V(G)$ with respect to $S$ is the $\beta$-vector $r_S(v) =(d(v, u_1), d(v, u_2), \ldots, d(v, u_{\beta}))$ and $r_{k,S}(v) =(d_k(v, u_1), d_k(v, u_2), \ldots, d_k(v, u_{\beta}))$, where $k$ is any positive integer. Note that a distance-1 resolving set and the distance-1 dimension, respectively, of $G$ corresponds to an \emph{adjacency resolving set} and the \emph{adjacency dimension} of $G$; notice $\dim_1(G)=\adim(G)$. Jannesari and Omoomi~\cite{adim} introduced adjacency dimension as a tool to study the metric dimension of lexicographic product graphs. 

In this paper, we study the distance-$k$ dimension of graphs. The paper is organized as follows. In Section~\ref{s:general}, we obtain some general results on distance-$k$ dimension of graphs. We prove that the maximum possible order of a graph $G$ with $\dim_d(G) = k$ is $(\lfloor \frac{2(d+1)}{3}\rfloor +1)^{k}+k \sum_{i = 1}^{\lceil \frac{d+1}{3}\rceil } (2i-1)^{k-1}$. It is easy to see that $\dim(G) \le \dim_k(G) \le \dim_1(G)$; we show that $\frac{\dim_k(G)}{\dim(G)}$ and $\frac{\dim_1(G)}{\dim_k(G)}$ can simultaneously be arbitrarily large with respect to $k$. 

In Section~\ref{s:char}, we prove characterization results for distance-$k$ dimension. For all positive integers $k \ge 1$, we characterize all connected graphs $G$ of order $n \ge 4$ for which $\dim_k(G)$ equals $n-2$ or $n-1$. In the case that $k = 1$, this solves the problem from \cite{broadcast} of characterizing the graphs $G$ with $\adim(G) = n-2$ when $G$ is connected. In Section~\ref{s:planar}, we examine the relationship between the distance-$k$ dimension and planarity of graphs. In Section~\ref{classes}, we determine $\dim_k(G)$ for some classes of graphs, including paths and cycles. 

In Section~\ref{s:delete}, we examine the effect of vertex or edge deletion on distance-$k$ dimension of graphs. Let $v$ and $e$, respectively, denote a vertex and an edge of a graph $G$. For any positive integer $k\ge 1$, we show that $\dim_k(G-v)-\dim_k(G)$ can be arbitrarily large (also see~\cite{broadcast} when $k=1$); for $k \ge 2$, we show that $\dim_k(G)-\dim_k(G-v)$ can be arbitrarily large, whereas it was shown in~\cite{broadcast} that $\dim_1(G)-\dim_1(G-v) \le 1$. It was shown in~\cite{broadcast} that $\dim_1(G)-1 \le \dim_1(G-e) \le \dim_1(G)+1$. We show that $\dim_2(G-e) \le \dim_2(G)+1$ and that $\dim_k(G-e) \le \dim_k(G)+2$ for $k \ge 3$. Moreover, in contrast to the case of distance-1 dimension, we show that $\dim_k(G)-\dim_k(G-e)$ can be arbitrarily large for $k\ge2$.

In this paragraph, we introduce some notation that we use in the paper. For $x\in V(G)$ and $S \subseteq V(G)$, let $d(x, S)=\min\{d(x,y) : y \in S\}$. The \emph{diameter}, $\diam(G)$, of $G$ is $\max\{d(x,y): x,y \in V(G)\}$. The \emph{join} of two graphs $H_1$ and $H_2$, denoted by $H_1+H_2$, is the graph obtained from the disjoint union of two graphs $H_1$ and $H_2$ by joining every vertex of $H_1$ with every vertex of $H_2$. We denote by $P_n$, $C_n$, $K_n$, and $K_{a,n-a}$ respectively, the path, the cycle, the complete graph, and the complete bipartite graph on $n$ vertices with one part of size $a$. Suppose $f(x)$ and $g(x)$ are two functions defined for all sufficiently large real numbers $x$. We write $f(x)=O(g(x))$ if there exist positive constants $N$ and $C$ such that $|f(x)| \le C |g(x)|$ for all $x >N$, $f(x)=\Omega(g(x))$ if $g(x)=O(f(x))$, and $f(x)=\Theta(g(x))$ if $f(x)=O(g(x))$ and $f(x)=\Omega(g(x))$.


\section{General bounds}\label{s:general}

In this section, we obtain some general bounds for distance-$k$ dimension of graphs. In order to state the results in this section, we define some terminology. The \emph{open neighborhood} of a vertex $v \in V(G)$ is $N(v)=\{u \in V(G) : uv \in E(G)\}$. For distinct $u,w\in V(G)$, if $N(u)-\{w\}=N(w)-\{u\}$, then  $u$ and $w$ are called \emph{twin vertices} of $G$. We begin with the following observations from \cite{Hernando,adim,br,aem} which we use in our proofs.

\begin{observation}\label{obs_twin}
Let $u$ and $w$ be twin vertices of a graph $G$, and let $k$ be a positive integer. Then 
\begin{itemize}
\item[(a)] \emph{\cite{Hernando}} $S \cap \{u,w\} \neq \emptyset$ for any resolving set $S$ of $G$;
\item[(b)] $S_k \cap \{u,w\} \neq \emptyset$ for any distance-$k$ resolving set $S_k$ of $G$.
\end{itemize}
\end{observation}

\begin{observation}\label{obs_bounds}\emph{\cite{adim, br}}
Let $G$ be a connected graph of order $n \ge 2$, and let $k$ and $k'$ be positive integers. Then
\begin{itemize}
\item[(a)] $\dim(G)\le \dim_k(G) \le \dim_1(G)$;
\item[(b)] if $k> k'$, then $\dim_k(G) \le \dim_{k'}(G)$.
\end{itemize}
\end{observation}

\begin{observation}\label{obs_diam}\emph{\cite{aem}}
Let $G$ be a connected graph with $\diam(G)=d$, and let $k$ be a positive integer. 
\begin{itemize}
\item[(a)]  If $d\in\{1,2\}$, then $\dim_k(G)=\dim(G)$ for any positive integer $k$.
\item[(b)]  If $d \ge 2$, then $\dim_k(G)=\dim_{d-1}(G)=\dim(G)$ for any $k\ge d-1$.
\end{itemize}
\end{observation}

In the next proof, we use a method similar to \cite{tree1} to obtain a general upper bound on $\dim_k(G)$ in terms of the diameter of $G$. In Section~\ref{s:char}, we use this result to characterize the connected graphs $G$ of order $n$ with $\dim_k(G) = n-2$ for all $k \ge 2$ and $n \ge 4$.

\begin{theorem}\label{upper_diam}
If $G$ is a connected graph of order $n \ge 2$ and diameter $d$, then $\dim_k(G) \le n-\min\{d, k+1\}$ for all $k \ge 1$.
\end{theorem}

\begin{proof}
Suppose that $u$ and $v$ are vertices in $G$ at distance $d$, and let $u = v_0, v_1, \ldots, v_d = v$ be a path of order $d+1$ with endpoints $u$ and $v$. If $d \le k+1$, then let $S = V(G) - \left\{v_1, \ldots, v_d\right\}$. Note that $d_k(v_0, v_i) = i$ for each $1 \le i \le d$, so $S$ is a distance-$k$ resolving set for $G$.

Otherwise $d > k+1$. In this case, let $S = V(G) - \left\{v_1, \ldots, v_{k+1}\right\}$. Note that $d_k(v_0, v_i) = i$ for each $1 \le i \le k+1$, so $S$ is a distance-$k$ resolving set for $G$.
\end{proof}

In Section \ref{classes}, we strengthen Theorem~\ref{upper_diam} after determining the value of $\dim_k(P_n)$. Next, we recall the following result by Hernando et al.

\begin{theorem}\emph{\cite{Hernando}}\label{dim_bound2}
Let $G$ be a connected graph of order $n$, $\diam(G)=d $, and $\dim(G)=\beta$. Then
$$n \le \left(\left\lfloor\frac{2d}{3}\right\rfloor+1\right)^{\beta}+\beta\sum_{i=1}^{\lceil\frac{d}{3}\rceil}(2i-1)^{\beta-1}.$$ 
\end{theorem}

Since $\dim_k(G)=\beta$ implies $\dim(G) \le \beta$ by Observation~\ref{obs_bounds}(a), we have the following 

\begin{corollary}\label{cor_bound1}
For any positive integer $k$ and for any connected graph $G$ with $\diam(G) = d$ and $\dim_k(G) = \beta$, 
$$|V(G)| \leq \left(\left\lfloor \frac{2d}{3}\right\rfloor +1\right)^{\beta}+\beta \sum_{i = 1}^{\lceil \frac{d}{3}\rceil} (2i-1)^{\beta-1}.$$
\end{corollary}

Using a method similar to the one in \cite{Hernando}, we find a sharp upper bound on the maximum possible order of a graph $G$ with $\dim_j(G) = k$. 

\begin{theorem}
The maximum possible order of a graph $G$ with $\dim_j(G) = k$ is $(\lfloor \frac{2(j+1)}{3}\rfloor +1)^{k}+k \sum_{i = 1}^{\lceil \frac{j+1}{3}\rceil } (2i-1)^{k-1}$.
\end{theorem}

\begin{proof}
First we prove the upper bound. Let $G$ be a graph with $\dim_j(G) = k$. Let $S$ be a distance-$j$ resolving set for $G$ and let $c \in [0, j]$ be an integer constant that will be chosen at the end. For each $v \in S$ and integer $i \in [0, c]$, define $N_{i}(v) = \left\{x \in V(G) : d_j(x, v) = i \right\}$.

Observe that $|d_j(x, u)-d_j(y, u)| \leq 2i$ for any two vertices $x, y\in N_{i}(v)$ and any vertex $u \in S$, so $d_j(x,v) = i$ and $d_j(x,t)$ has $2i+1$ possible values for each $t \in S$ such that $t \neq v$. Thus $|N_{i}(v)| \leq (2i+1)^{k-1}$. 

Consider $x \in V(G)$ such that $x \not \in N_{i}(v)$ for all $i \in [0, c]$ and $v \in S$, i.e., $d_j(x,v) \geq c+1$ for all $v \in S$. Since $S$ is a distance-$j$ resolving set for $G$, there are at most $(j-c+1)^k$ such vertices. Thus

\[ |V(G)| \leq (j-c+1)^{k}+k\sum_{i = 0}^{c} (2i+1)^{k-1} \]

Setting $c = \lceil \frac{j+1}{3} \rceil -1$ gives the upper bound. To see that the upper bound is sharp, note that the construction in \cite{Hernando} of a graph $G$ of maximum order with diameter $j+1$ and $\dim(G) = k$ must also have $\dim_j(G) = k$ and the same order as the bound we just obtained.
\end{proof}

\begin{remark}
In \cite{broadcast}, there is a simple construction of a graph $G$ with $\dim_1(G) = k$ of maximum order $k+2^k$. For the $j = 2$ case, we also found a simple construction of a graph $G$ with $\dim_2(G) = k$ of maximum order $k+3^k$, which is similar to a construction in \cite{mdapa}. Start with $k$ copies of $K_2$, each on vertices $a_i$ and $b_i$ for $i = 1, \dots, k$. Let $c_j$ for $j = 1, \dots, 3^k$ be labeled with a ternary string. Add an edge from $c_j$ to $a_i$ if the $i^{th}$ digit of $c_j$ is $0$. Add an edge from $c_j$ to $b_i$ if the $i^{th}$ digit of $c_j$ is $1$. Let $S = \left\{a_1, \dots, a_k\right\}$. Remove any $c_j$ with the same distance-$2$ vector as $b_i$ with respect to $S$ for each $i = 1, \dots, k$. The resulting graph $G$ has order $k+3^k$, and $S$ is a distance-$2$ resolving set, so $\dim_2(G) = k$. 
\end{remark}

It was shown in~\cite{linegraph} that metric dimension is not a monotone parameter on subgraph inclusion. Moreover, it was shown in~\cite{broadcast} that, for two graphs $G$ and $H$ with $H \subset G$, $\frac{\dim(H)}{\dim(G)}$ and $\frac{\dim_1(H)}{\dim_1(G)}$ can be arbitrarily large.

Following~\cite{broadcast}, for $m \ge 3$, let $H=K_{\frac{m(m+1)}{2}}$; let $V(H)$ be partitioned into $V_1, V_2, \ldots, V_m$ such that $V_i=\{w_{i,1}, w_{i,2}, \ldots, w_{i,i}\}$ with $|V_i|=i$, where $i \in \{1,2,\ldots,m\}$. Let $G$ be the graph obtained from $H$ and $m$ isolated vertices $u_1,u_2, \ldots, u_m$ such that, for each $i \in \{1,2,\ldots, m\}$, $u_i$ is joined by an edge to each vertex of $V_i \cup (\cup_{j=i+1}^{m}\{w_{j,i}\})$. Since $\diam(H)=1$ and $\diam(G)=2$, by Observation~\ref{obs_diam}(a), $\dim(H)=\dim_k(H)$ and $\dim(G)=\dim_k(G)$ for every positive integer $k$. Note that $H\subset G$, $\dim(H) =\frac{m(m+1)}{2}-1$ by Theorem~\ref{dim_characterization}(c), and $\dim(G) \le m$ since $\{u_1, u_2, \ldots, u_m\}$ forms a resolving set of $G$. So, $\frac{\dim_k(H)}{\dim_k(G)}=\frac{\dim(H)}{\dim(G)} \ge \frac{m^2+m-2}{2m}$ for every positive integer $k$, which implies the following.

\begin{corollary}
For all positive integers $k$ and $N$, there exist connected graphs $G$ and $H$ such that $H \subset G$ and $\frac{\dim_k(H)}{\dim_k(G)} > N$.
\end{corollary}

Next, in view of Observation~\ref{obs_bounds}(a), we show that $\frac{\dim_k(G)}{\dim(G)}$ and $\frac{\dim_1(G)}{\dim_k(G)}$ can be arbitrarily large with respect to $k$; thus, $\dim_k(G)-\dim(G)$ and $\dim_1(G)-\dim_k(G)$ can be arbitrarily large with respect to $k$.

\begin{proposition}\emph{\cite{grid}}\label{dim_grid}
For the grid graph $G=P_m \times P_n$ ($m,n \ge 2$), $\dim(G)=2$.
\end{proposition}

\begin{theorem}\emph{\cite{broadcast}}\label{adim_grid}
For $m \ge 2$, let $G=P_m \times P_m$. Then $\dim_1(G)=\Theta(m^2)$; thus $\frac{\dim_1(G)}{\dim(G)}$ can be arbitrarily large with respect to $m$.
\end{theorem}

\begin{theorem}\label{comp_kdim}
For any positive integer $k>1$, let $G=P_{k^2} \times P_{k^2}$. Then $\dim_k(G)=\Theta(k^2)$, and thus $\frac{\dim_k(G)}{\dim(G)}$ and $\frac{\dim_1(G)}{\dim_k(G)}$ can simultaneously be arbitrarily large with respect to $k$.
\end{theorem}

\begin{proof}
First, observe that $\dim(G)=2$ by Proposition~\ref{dim_grid}, and $\dim_1(G)=\Theta(k^4)$ by Theorem~\ref{adim_grid}. Next, we show that $\dim_k(G)=\Theta(k^2)$. Since any distance-$k$ resolving set of $G$ must contain at least one vertex from every $P_{2k+1} \times P_{2k+1}$ subgraph of $G$ except for at most one such subgraph, $\dim_k(G) \ge \frac{k^4}{(2k+1)^2}-1$. On the other hand, if the grid graph $P_{k^2} \times P_{k^2}$ is drawn in the $xy$-plane with the four corners at $(1,1)$, $(k^2,1)$, $(1,k^2)$ and $(k^2,k^2)$ and with horizontal/vertical edges of equal lengths, then $[\cup_{j=0}^{k+1} \cup_{i=0}^{k+1}\{(1+(k-1)i, 1+(k-1)j)\}] \cup [\cup_{j=0}^{k}\cup_{i=0}^{k}\{(\lceil\frac{k}{2}\rceil+(k-1)i, \lceil\frac{k}{2}\rceil+(k-1)j)\}]$ forms a distance-$k$ resolving set of $G$, and hence $\dim_k(G) \le (k+2)^2+(k+1)^2 < 2(k+2)^2$. So, $\dim_k(G)=\Theta(k^2)$. Therefore, $\frac{\dim_k(G)}{\dim(G)}$ and $\frac{\dim_1(G)}{\dim_k(G)}$ can  simultaneously be arbitrarily large with respect to $k$.~\hfill
\end{proof}

\section{Characterizing graphs by their distance-$k$ dimension}\label{s:char}

It is known that, for any connected graph $G$ of order at least two, $1 \le \dim(G) \le |V(G)|-1$ (see~\cite{tree1}) and $1 \le \dim_1(G) \le |V(G)|-1$ (see~\cite{adim}). We recall some characterization results on metric dimension and distance-$1$ dimension, before proving characterization results about distance-$k$ dimension.

\begin{theorem}\emph{\cite{tree1}}\label{dim_characterization}
Let $G$ be a connected graph of order $n \ge 2$. Then 
\begin{itemize}
\item[(a)] $\dim(G)=1$ if and only if $G=P_n$;
\item[(b)] for $n \ge 4$, $\dim(G)=n-2$ if and only if $G=K_{s,t}$ ($s,t \ge 1$), $G=K_s+\overline{K_t}$ ($s\ge1, t\ge2$), or $G=K_s+(K_1 \cup K_t)$ ($s,t \ge 1$), where $\overline{H}$ denotes the complement of a graph $H$;
\item[(c)] $\dim(G)=n-1$ if and only if $G=K_n$.
\end{itemize}
\end{theorem}

\begin{theorem}\emph{\cite{adim}}\label{adj_characterization}
Let $G$ be a connected graph of order $n \ge 2$. Then
\begin{itemize}
\item[(a)] $\dim_1(G)=1$ if and only if $G \in \{P_2, P_3\}$;
\item[(b)] $\dim_1(G)=n-1$ if and only if $G=K_n$.
\end{itemize}
\end{theorem}

More generally, the characterization of graphs $G$ with $\dim_1(G)=\beta$ is provided in~\cite{broadcast} (this includes disconnected graphs). Given any graph $G_1$ on $\beta$ vertices $v_1, \dots, v_{\beta}$ and $G_2$ on $2^{\beta}$ vertices $\left\{u_b\right\}_{b \in \left\{0,1\right\}^{\beta}}$, define the graph $B(G_1, G_2)$ to be obtained by connecting $v_i$ and $u_b$ if and only if the $i^{th}$ digit of $b$ is $1$. Moreover, define $\mathcal{B}(G_1, G_2)$ to be the family of induced subgraphs of $B(G_1, G_2)$ that contain every vertex in $G_1$. Finally, define $\mathcal{H}_0 = \emptyset$ and, for each positive integer $\beta$, define $\mathcal{H}_{\beta}$ to be the family of graphs obtained from taking the union of $\mathcal{B}(G_1, G_2)$ over all graphs $G_1$ with $j$ vertices $v_1, \dots, v_j$ and $G_2$ with $2^j$ vertices $\left\{u_b\right\}_{b \in \left\{0,1\right\}^j}$, for each $1 \leq j \leq \beta$.

\begin{theorem}\emph{\cite{broadcast}}\label{adimch}
For each $\beta \geq 1$, the set of graphs $G$ with $\dim_1(G) = \beta$ is $\mathcal{H}_{\beta} - \mathcal{H}_{\beta-1}$ up to isomorphism.
\end{theorem}

By the definition of $\dim_k(G)$, Observation~\ref{obs_bounds}(a), and Theorems~\ref{dim_characterization} and~\ref{adj_characterization}, we have the following

\begin{corollary}\label{kdim_characterization}
Let $G$ be a connected graph of order $n\ge2$, and let $k$ be any positive integer. Then $1\le\dim_k(G)\le n-1$, and 
\begin{itemize}
\item[(a)] \emph{\cite{aem}} $\dim_k(G)=1$ if and only if $G \in \cup_{i=2}^{k+2}\{P_i\}$,
\item[(b)] $\dim_k(G)=n-1$ if and only if $G=K_n$.
\end{itemize}
\end{corollary}

In the next result, we characterize the connected graphs $G$ of order $n$ with $\dim_k(G) = n-2$ for each $k \ge 2$. What is interesting is that these are exactly the same connected graphs $G$ of order $n$ for which $\dim(G) = n-2$. 

\begin{theorem}\label{n2charj2}
Let $G$ be a connected graph of order $n \ge 4$, and let $k \ge 2$. Then $\dim_k(G) = n-2$ if and only if $G = K_{s, t}$ with $s, t \ge 1$, $G = K_s + \overline{K_t}$ with $s \ge 1$ and $t \ge 2$, or $G = K_s + (K_1 \cup K_t)$ with $s, t \ge 1$. 
\end{theorem}

\begin{proof}
First, note that all of the graphs $G$ in the statement of the theorem have $\dim_k(G) = n-2$. This follows immediately from Observation \ref{obs_diam}(a) and the paper \cite{tree1}, since all graphs $G$ in the statement of the theorem have diameter $2$, and each of these graphs $G$ have $\dim(G) = n-2$ \cite{tree1}. This proves the backward implication of the biconditional.

Now we prove the forward implication. Suppose that $G$ is a connected graph of order $n \ge 4$ with $\dim_k(G) = n-2$. Since $k \ge 2$, the diameter of $G$ must be $2$ by Theorem \ref{upper_diam} and Corollary \ref{kdim_characterization}(b). Thus $\dim(G) = n-2$ by Observation~\ref{obs_diam}(a). Thus the result follows by Theorem~\ref{dim_characterization}(b).
\end{proof}

It is also interesting that this similarity between $\dim(G)$ and $\dim_k(G)$ breaks at $k = 1$. We just showed for $k \ge 2$ that the connected graphs $G$ of order $n \ge 4$ with $\dim(G) = n-2$ are the same as the connected graphs $G$ of order $n$ with $\dim_k(G) = n-2$. However when $k = 1$, observe that $\dim_1(P_4) = 2$ but $\dim(P_4) = 1$, so there exists a connected graph $G$ of order $n = 4$ with $\dim_1(G) = n-2$ and $\dim(G) < n-2$. In the next result, we show that this is the only connected graph $G$ of order $n$ for which $\dim_1(G) = n-2$ and $\dim(G) < n-2$. The next theorem answers an open problem from \cite{broadcast} in the case that $G$ is connected. 

\begin{theorem}\label{n2charj1}
Let $G$ be a connected graph of order $n \ge 4$. Then $\dim_1(G) = n-2$ if and only if $G = K_{s, t}$ with $s, t \ge 1$, $G = K_s + \overline{K_t}$ with $s \ge 1$ and $t \ge 2$, $G = K_s + (K_1 \cup K_t)$ with $s, t \ge 1$, or $G = P_4$. 
\end{theorem}

\begin{proof}
First, note that all of the graphs $G$ in the statement of the theorem have $\dim_1(G) = n-2$. For all of the graphs except for $P_4$, this follows immediately from Observation \ref{obs_diam}(a) and the paper \cite{tree1}, since all graphs $G$ in the statement of the theorem besides $P_4$ have diameter $2$, and each of these graphs $G$ have $\dim(G) = n-2$ \cite{tree1}. In the case of $P_4$, clearly $\dim_1(P_4) = 2$. This proves the backward implication of the biconditional.

Now we prove the forward implication. Suppose that $G$ is a connected graph of order $n \ge 4$ with $\dim_1(G) = n-2$. Note that $G$ must have diameter at most $3$, or else $\dim_1(G) \le n-3$. To see why this is true, note that if $G$ had two vertices $u$ and $v$ with $d(u, v) = 4$, then there would exist vertices $x, y, z$ in $G$ such that $u, x, y, z, v$ is an induced path of order $5$ in $G$. Then $V(G) - \left\{u, y, v\right\}$ would be a distance-$1$ resolving set for $G$, contradicting the fact that $\dim_1(G) = n-2$, so $G$ has diameter at most $3$. 

For the first case, suppose that $G$ has diameter $3$, so there exist vertices $u$ and $v$ in $G$ with $d(u, v) = 3$. Since $d(u, v) = 3$, there must exist vertices $x, y \in V(G)$ such that $u, x, y, v$ form an induced path of order $4$ in $G$. 

For contradiction, assume that $G$ has another vertex besides $u, v, x, y$. Let $t$ be a vertex in $G$ that is not in the copy of $P_4$ such that $t$ is adjacent to some vertex in the copy of $P_4$. Note that $t$ must be adjacent to $x$ or $y$. To see why this is true, note that if $t$ was only adjacent to one of $u$ or $v$ and neither of $x$ nor $y$, then $G$ would have diameter at least $4$. If $t$ was adjacent to both $u$ and $v$, then $d(u, v) \le 2$, a contradiction. Thus, $t$ is adjacent to at most one of $u$ or $v$, and at least one of $x$ or $y$. Without loss of generality, suppose that $t$ is not adjacent to $v$.

Since $t$ is adjacent to $x$ or $y$, and $t$ is not adjacent to $v$, there are several cases to consider. For the first case, suppose that $t$ is only adjacent to a single vertex among $u, v, x, y$. This vertex must be $x$ or $y$. Without loss of generality, let $t$ be adjacent to $x$. Then $V(G)-\left\{x, t, v\right\}$ is a distance-$1$ resolving set for $G$, so $\dim_1(G) \le n-3$, a contradiction. 

Now suppose that $t$ is adjacent to two vertices among $u, v, x, y$. We know $t$ is not adjacent to $v$, so either $t$ is adjacent to $u$ and $x$, $t$ is adjacent to $u$ and $y$, or $t$ is adjacent to $x$ and $y$. If $t$ is adjacent to $u$ and $x$, then $V(G)-\left\{x, t, v\right\}$ is a distance-$1$ resolving set for $G$, so $\dim_1(G) \le n-3$, a contradiction. If $t$ is adjacent to $u$ and $y$, then $V(G)-\left\{u, x, y\right\}$ is a distance-$1$ resolving set for $G$, so $\dim_1(G) \le n-3$, a contradiction.  If $t$ is adjacent to $x$ and $y$, then $V(G)-\left\{v, x, y\right\}$ is a distance-$1$ resolving set for $G$, so $\dim_1(G) \le n-3$, a contradiction. 

Now suppose that $t$ is adjacent to three vertices among $u, v, x, y$. Since $t$ is not adjacent to $v$, $t$ must be adjacent to $u$, $x$, and $y$.  Then $V(G)-\left\{v, x, y\right\}$ is a distance-$1$ resolving set for $G$, so $\dim_1(G) \le n-3$, a contradiction. This covers all of the possible cases, since $t$ is not adjacent to $v$. Thus the only vertices in $G$ are $u, v, x, y$, and these vertices form an induced path, so $G$ is $P_4$ in the case that $G$ has diameter $3$.

Now we can assume that $G$ has diameter $2$, since we know that $G$ has diameter at most $3$, and we have already considered the case when $G$ has diameter $3$. Thus $\dim(G) = n-2$ by Observation~\ref{obs_diam}(a) and the result follows by Theorem~\ref{dim_characterization}(b).
\end{proof}

The last few results lead to a natural question. Note that the following problem is hard in general, since computing $\dim_k(G)$ is NP-hard for any positive integer $k$ \cite{eyr, NP, Juan, tree2}.

\begin{question}
Can we characterize connected graphs $G$ of order $n \ge 4$ such that $\dim_k(G)=\beta$, where $k \ge 1$ and $\beta \in \{2,3,\ldots, n-3\}$?
\end{question}

\section{Planarity and the distance-$k$ dimension}\label{s:planar}

Next, we consider the relation between $\dim_k(G)$ and planarity of $G$. A graph is \emph{planar} if it can be drawn in a plane without any edge crossing. For two graphs $G$ and $H$, $H$ is called a \emph{minor} of $G$ if $H$ can be obtained from $G$ by vertex deletion, edge deletion, or edge contraction. We recall some known results on metric dimension and its variations in conjunction with planarity of a graph.

\begin{theorem}\emph{\cite{wagner}}\label{minor}
A graph $G$ is planar if and only if neither $K_5$ nor $K_{3,3}$ is a minor of $G$.
\end{theorem}

\begin{theorem}\emph{\cite{tree2}}
\begin{itemize}
\item[(a)] A graph $G$ with $\dim(G)=2$ cannot have $K_5$ or $K_{3,3}$ as a subgraph. 
\item[(b)] There exists a non-planar graph $G$ with $\dim(G)=2$.
\end{itemize}
\end{theorem}

\begin{theorem}\emph{\cite{broadcast}}\label{adim_planar} 
\begin{itemize}
\item[(a)] If $\dim_1(G)=2$, then $G$ is planar; see Figure~\ref{fig_planarity}(a) for graphs $G$ with $\dim_1(G)=2$.
\item[(b)] For each integer $\beta \ge 3$, there exists a non-planar graph $G$ with $\dim_1(G)=\beta$.
\end{itemize}
\end{theorem}

Another variant of metric dimension, called connected metric dimension was introduced in~\cite{cdim}. A resolving set $S$ of $G$ is called a \emph{connected resolving set} of $G$ if the subgraph of $G$ induced by $S$ is connected, and the \emph{connected metric dimension}, $\cdim(G)$, of $G$ is the minimum cardinality over all connected resolving sets of $G$. For the characterization of graphs $G$ with $\cdim(G)=2$, see~\cite{cdim}.

\begin{theorem}\emph{\cite{cdim}} 
\begin{itemize}
\item[(a)] If $\cdim(G)=2$, then $G$ is planar. However, there exists a non-planar graph $G$ with $\dim(G)=2$ and $\cdim(G)>2$; see Figure~\ref{fig_planarity}(b). 
\item[(b)] For each integer $\beta \ge 3$, there exists a non-planar graph $G$ with $\cdim(G)=\beta$. 
\end{itemize}
\end{theorem}

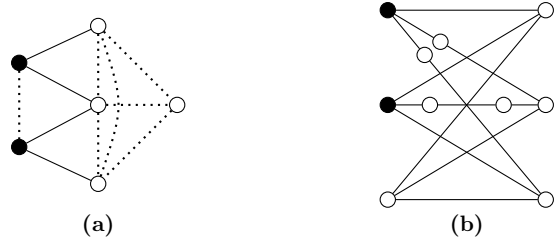
\begin{figure}[ht]
\centering
\begin{tikzpicture}[scale=.7, transform shape]

\node [draw, fill=black, shape=circle, scale=.8] (u1) at  (0, 0.8) {};
\node [draw, fill=black, shape=circle, scale=.8] (u2) at  (0, -0.8) {};
\node [draw, shape=circle, scale=.8] (1) at  (1.5, 1.5) {};
\node [draw, shape=circle, scale=.8] (2) at  (1.5, 0) {};
\node [draw, shape=circle, scale=.8] (3) at  (1.5, -1.5) {};
\node [draw, shape=circle, scale=.8] (4) at  (3, 0) {};

\node [draw, fill=black, shape=circle, scale=.8] (v1) at  (7, 1.8) {};
\node [draw, fill=black, shape=circle, scale=.8] (v2) at  (7, 0) {};
\node [draw, shape=circle, scale=.8] (11) at  (7, -1.8) {};
\node [draw, shape=circle, scale=.8] (22) at  (10, 1.8) {};
\node [draw, shape=circle, scale=.8] (33) at  (10, 0) {};
\node [draw, shape=circle, scale=.8] (44) at  (10, -1.8) {};
\node [draw, shape=circle, scale=.8] (5) at  (8, 1.2) {};
\node [draw, shape=circle, scale=.8] (6) at  (7.7, 0.95) {};
\node [draw, shape=circle, scale=.8] (7) at  (7.8, 0) {};
\node [draw, shape=circle, scale=.8] (8) at  (9.2, 0) {};

\node [scale=1.1] at (1.5,-2.3) {\textbf{(a)}};
\node [scale=1.1] at (8.5,-2.3) {\textbf{(b)}};

\draw(1)--(u1)--(2)--(u2)--(3);
\draw[thick,dotted](u1)--(u2);
\draw[thick,dotted](1)--(2)--(3).. controls (2,0) .. (1);\draw[thick,dotted](1)--(4)--(3);\draw[thick, dotted](2)--(4);

\draw(22)--(v1)--(5)--(33)--(8)--(7)--(v2)--(44)--(11);
\draw(v2)--(22)--(11);
\draw(v1)--(6)--(44);\draw(11)--(33);

\end{tikzpicture}
\caption{\small (a)~\cite{broadcast} The graphs $G$ satisfying $\dim_1(G)=2$, where black vertices must be present, a solid edge must be present whenever the two vertices incident to the solid edge are in the graph, but a dotted edge is not necessarily present; (b)~\cite{cdim} A non-planar graph $G$ with $\dim(G)=2$ and $\cdim(G)=3$, where black vertices form a minimum resolving set of $G$.}\label{fig_planarity}
\end{figure}

Now, we consider the relation between distance-$k$ dimension and planarity of graphs. 

\begin{theorem} \begin{itemize}
\item[(a)] For each $k \ge 2$, there is a non-planar connected graph $G$ with $\dim_k(G)=2$. 
\item[(b)] For each $k \ge 1$ and $\beta \ge 3$, there is a non-planar connected graph $G$ with $\dim_k(G) = \beta$. 
\end{itemize}
\end{theorem}

\begin{proof}
For the first part, an example of a non-planar graph $G$ with $\dim_k(G)=2$ is given in Figure~\ref{fig_planarity}(b), where black vertices form a minimum distance-$k$ resolving set of $G$ for each $k \ge 2$. 

For the second part, let $G$ be a graph obtained from $K_{m+2}$ ($m \ge 3$) by subdividing exactly one edge once; then $G$ is non-planar by Theorem~\ref{minor}. It was shown in~\cite{cdim} that $\cdim(G)=\dim(G)=m$. Since $\diam(G)=2$, $\dim_k(G)=\dim(G)=m$ by Observation~\ref{obs_diam}(a).~\hfill
\end{proof}


\section{The distance-$k$ dimension of some classes of graphs}\label{classes}

In this section, we determine $\dim_k(G)$ for some classes of graphs. First, we consider graphs $G$ with $\diam(G) \le 2$. For two graphs $H_1$ and $H_2$, $\diam(H_1+H_2) \le 2$; thus, by Observation~\ref{obs_diam}(a), $\dim(H_1+H_2)=\dim_k(H_1+H_2)$ for any positive integer $k$.

\begin{theorem}\emph{\cite{wheel1, wheel2}}\label{dim_wheel}
For $n \ge 3$, 
\begin{equation*}
\dim(C_n+K_1)=\left\{
\begin{array}{ll}
3 & \mbox{ if } n \in \{3,6\},\\
\lfloor \frac{2n+2}{5}\rfloor & \mbox{ otherwise.} 
\end{array}\right.
\end{equation*}
\end{theorem}

\begin{theorem}\emph{\cite{fan}}\label{dim_fan}
For $n \ge 1$, 
\begin{equation*}
\dim(P_n+K_1)=\left\{
\begin{array}{ll}
1 & \mbox{ if } n=1,\\
2 & \mbox{ if } n \in \{2,3\},\\
3 & \mbox{ if } n=6,\\
\lfloor \frac{2n+2}{5}\rfloor & \mbox{ otherwise.} 
\end{array}\right.
\end{equation*}
\end{theorem}

By Observation~\ref{obs_diam}(a) and Theorems~\ref{dim_wheel} and~\ref{dim_fan}, we have the following.

\begin{corollary}\label{cor_wheel1}
For any positive integer $k$ and for $n \ge 3$, 
\begin{equation*}
\dim_k(C_n+K_1)=\left\{
\begin{array}{ll}
3 & \mbox{ if } n \in \{3,6\},\\
\lfloor \frac{2n+2}{5}\rfloor & \mbox{ otherwise.} 
\end{array}\right.
\end{equation*}
\end{corollary}

\begin{corollary}\label{cor_fan1}
For any positive integer $k$ and for $n \ge 1$, 
\begin{equation*}
\dim_k(P_n+K_1)=\left\{
\begin{array}{ll}
1 & \mbox{ if } n=1,\\
2 & \mbox{ if } n \in \{2,3\},\\
3 & \mbox{ if } n=6,\\
\lfloor \frac{2n+2}{5}\rfloor & \mbox{ otherwise.} 
\end{array}\right.
\end{equation*}
\end{corollary}

The metric dimension of complete multi-partite graphs was determined in~\cite{kpartite}.

\begin{theorem}\emph{\cite{kpartite}}\label{adim_kpartite}
For $m \ge 2$, let $G=K_{a_1, a_2, \ldots, a_m}$ be a complete $m$-partite graph of order $n=\sum_{i=1}^{m}a_i$. Let $s$ be the number of partite sets of $G$ consisting of exactly one element. Then 
\begin{equation*}
\dim(G)=\left\{
\begin{array}{ll}
n-m & \mbox{ if } s=0,\\
n-m+s-1 & \mbox{ if } s \neq 0.
\end{array}\right.
\end{equation*}
\end{theorem}

As an immediate consequence of Observation~\ref{obs_diam}(a) and Theorem~\ref{adim_kpartite}, we have the following.

\begin{corollary}\label{kdim_kpartite}
For $m \ge 2$, let $G=K_{a_1, a_2, \ldots, a_m}$ be a complete $m$-partite graph of order $n=\sum_{i=1}^{m}a_i$. Let $s$ be the number of partite sets of $G$ consisting of exactly one element. Then, for any positive integer $k$,  
\begin{equation*}
\dim_k(G)=\left\{
\begin{array}{ll}
n-m & \mbox{ if } s=0,\\
n-m+s-1 & \mbox{ if } s \neq 0.
\end{array}\right.
\end{equation*}
\end{corollary}

Now, we recall the metric dimension of the Petersen graph.

\begin{theorem}\emph{\cite{petersen}}\label{dim_petersen}
For the Petersen graph $\mathcal{P}$, $\dim(\mathcal{P})=3$.
\end{theorem}

Since $\diam(\mathcal{P})=2$, Observation~\ref{obs_diam}(a) and Theorem~\ref{dim_petersen} imply the following

\begin{corollary}
For the Petersen graph $\mathcal{P}$ and for any positive integer $k$, $\dim_k(\mathcal{P})=3$.
\end{corollary}

Next, we determine the distance-$k$ dimension of cycles. We recall the following results.

\begin{proposition}\emph{\cite{unicyclic}}\label{dim_cycle}
For $n \ge 3$, $\dim(C_n)=2$.
\end{proposition}

\begin{proposition}\emph{\cite{adim}}\label{adim_c}
For $n \ge 4$, $\dim_1(C_n)=\lfloor \frac{2n+2}{5}\rfloor$.
\end{proposition}

Following~\cite{wheel1}, let $M$ be a set of at least two vertices of $C_n$, let $u_i$ and $u_j$ be distinct vertices of $M$, and let $P$ and $P'$ denote the two distinct $u_i-u_j$ paths determined by $C_n$. If either $P$ or $P'$, say $P$, contains only two vertices of $M$ (namely, $u_i$ and $u_j$), then we refer to $u_i$ and $u_j$ as \emph{neighboring vertices} of $M$ and the set of vertices of $P$ that belong to $C_n -\{u_i,u_j\}$ as the \emph{gap} of $M$ (determined by $u_i$ and $u_j$). The two gaps of $M$ determined by a vertex of $M$ and its two neighboring vertices of $M$ are called \emph{neighboring gaps}. Note that, $M$ has $r$ gaps if $|M|=r$, where some of the gaps may be empty.

\begin{lemma}\label{lem_cycle}
For a positive integer $k$, let $M_k$ be a minimum distance-$k$ resolving set of $C_n$ for $n \ge 2k+3$. Then
\begin{itemize}
\item[(a)] Every gap of $M_k$ contains at most $2k+1$ vertices. Moreover, at most one gap of $M_k$ contains $2k+1$ vertices.
\item[(b)] If a gap of $M_k$ contains at least $k+1$ vertices, then any neighboring gaps contain at most $k$ vertices.
\end{itemize}
\end{lemma}

\begin{proof}
For a positive integer $k$, let $M_k$ be a minimum distance-$k$ resolving set of $C_n$ for $n \ge 2k+3$.

(a) If there is a gap of $M_k$ containing $2k+2$ consecutive vertices $u_{j}, u_{j+1}, \ldots, u_{j+2k+1}$ of $C_n$, where $0 \le j \le n-1$ and the subscript is taken modulo $n$, then $r_{k,M_k}(u_{j+k})=r_{k, M_k}(u_{j+k+1})$. If there exist distinct two gaps $\{u_{p}, u_{p+1}, \ldots, u_{p+2k}\}$ and $\{u_q, u_{q+1}, \ldots, u_{q+2k}\}$ of $M_k$, then $r_{k, M_k}(u_{p+k})=r_{k, M_k}(u_{q+k})$.

(b) Suppose a gap of $M_k$ contains at least $k+1$ vertices, and one of its neighboring gap contains more than $k$ vertices. Then there exist $2k+3$ consecutive vertices $u_j, u_{j+1}, \ldots, u_{j+2k+2}$ of $C_n$ such that $M_k \cap \{u_j, u_{j+1}, \ldots, u_{j+2k+2}\}=\{u_{j+k+1}\}$ and $r_{k, M_k}(u_{j+k})=r_{k, M_k}(u_{j+k+2})$.~\hfill
\end{proof}

\begin{theorem}\label{kdim_cycle}
Let $n\ge 3$ and let $k$ be any positive integer. 
\begin{itemize}
\item[(a)] If $n \le 3k+3$, then $\dim_k(C_n)=2$.
\item[(b)] If $n \ge 3k+4$, then
\begin{equation*}
\dim_{k}(C_n)=\left\{
\begin{array}{ll}
\lfloor\frac{2n+3k-1}{3k+2}\rfloor & \mbox{ if } n \equiv 0,1,\ldots, k+2 \pmod{(3k+2)},\\ 
\lfloor\frac{2n+4k-1}{3k+2}\rfloor & \mbox{ if } n \equiv k+3,\ldots,  \lceil \frac{3k+5}{2}\rceil-1 \pmod{(3k+2)},\\ 
\lfloor\frac{2n+3k-1}{3k+2}\rfloor & \mbox{ if } n \equiv \lceil \frac{3k+5}{2}\rceil,\ldots, 3k+1 \pmod{(3k+2)}.
\end{array}\right.
\end{equation*}
\end{itemize}
\end{theorem}

\begin{proof}
Let $C_n$ be given by $u_0,u_1, \ldots, u_{n-1}, u_0$ for $n \ge 3$, and let $k$ be a positive integer. 

(a) Let $n \le 3k+3$. Since $\{u_0, u_{\alpha}\}$, where $\alpha=\min\{2k+2, n-1\}$, forms a distance-$k$ resolving set of $C_n$, $\dim_k(C_n) \le 2$. By Corollary~\ref{kdim_characterization}(a), $\dim_k(C_n) \ge 2$. Thus, $\dim_k(C_n)=2$ for $n \le 3k+3$.

(b) Let $n \ge 3k+4$; then $\dim_k(C_n)\ge3$. Since $k$ is a positive integer, we must have $1 \le k \le \lfloor\frac{n}{2}\rfloor-2$. 

First, we show that $\dim_k(C_n) \ge \lfloor\frac{2n+3k-1}{3k+2}\rfloor$; moreover, we show that $\dim_k(C_n) \ge \lfloor\frac{2n+4k-1}{3k+2}\rfloor$ if $n \equiv k+3,\ldots,  \lceil \frac{3k+5}{2}\rceil-1 \pmod{(3k+2)}$. Let $S_k$ be a minimum distance-$k$ resolving set of $C_n$. If $|S_k|=2\ell$ for some positive integer $\ell$, then at most $\ell$ gaps contain more than $k$ vertices by Lemma~\ref{lem_cycle}(b), and those $\ell$ gaps contain at most $2k$ vertices except possibly one gap containing $2k+1$ vertices by Lemma~\ref{lem_cycle}(a); thus, the number of vertices belonging to the gaps of $S_k$ is at most $3k\ell+1$, and hence $n-2\ell \le 3k\ell+1$, which implies $|S_k|=2\ell \ge \lceil \frac{2n-2}{3k+2}\rceil=\lfloor \frac{2n+3k-1}{3k+2}\rfloor$. If $|S_k|=2\ell+1$ for some positive integer $\ell$, then at most $\ell$ gaps contain more than $k$ vertices by Lemma~\ref{lem_cycle}(b), and those $\ell$ gaps contain at most $2k$ vertices except possibly one gap containing $2k+1$ vertices by Lemma~\ref{lem_cycle}(a); thus, the number of vertices belonging to the gaps of $S_k$ is at most $3k\ell+k+1$, and hence $n-(2\ell+1) \le 3k\ell+k+1$, which implies $|S_k|=2\ell+1 \ge \lceil\frac{2n+k-2}{3k+2}\rceil=\lfloor \frac{2n+4k-1}{3k+2}\rfloor \ge \lfloor \frac{2n+3k-1}{3k+2}\rfloor$. 

Now, suppose $n=(3k+2)x+j$, where $x \ge 1$ and $k+3 \le j \le \lceil \frac{3k+5}{2}\rceil-1$; notice $k \ge 2$. Then $|S_k| = 2x+2$. To see why $|S_k| \le 2x+2$, for every path $P$ consisting of $3k+2$ vertices on $C_n$, we put exactly two vertices of $P$ in $S_k$ including one end vertex, say $v$, of $P$ and another vertex of $P$ at distance $2k+1$ from $v$ in $S_k$. Since $n\le (3k+2)x+\lceil \frac{3k+5}{2}\rceil-1 \le (3k+2)x+3k+2$, $|S_k| \le 2x+2$. To see why $|S_k| \ge 2x+2$, first observe that $|S_k|\ge 2x+1$ follows from the lower bounds that we proved in the last paragraph. However if $|S_k|=2x+1$, then we proved that $|S_k| \ge \lceil\frac{2n+k-2}{3k+2}\rceil \ge 2x+2$, giving a contradiction. Since $|S_k|=2x+2$, we have $|S_k|=\lfloor \frac{2n+4k-1}{3k+2}\rfloor$ in this case.

Now we show that $\dim_k(C_n) \le \lfloor\frac{2n+3k-1}{3k+2}\rfloor$ if $n \equiv 0,1,\ldots, k+2 \pmod{(3k+2)}$ or $n \equiv \lceil \frac{3k+5}{2}\rceil,\ldots, 3k+1 \pmod{(3k+2)}$. 

Case 1: $n=(3k+2)x+j$, where $x \ge 1$ and $0 \le j \le 1$. Note that $\lfloor\frac{2n+3k-1}{3k+2}\rfloor=2x$. Let $S_k=\{u_{0}, u_{2k+2}\} \cup (\cup_{i=1}^{x-1}\{u_{(3k+2)i+1}, u_{(3k+2)i+2k+2}\})$. Then $S_k$ is a distance-$k$ resolving set of $C_n$ with $|S_k|=2x$. So, $\dim_k(C_n)\le |S_k|=2x=\lfloor\frac{2n+3k-1}{3k+2}\rfloor$.

Case 2: $n=(3k+2)x+j$, where $x \ge 1$ and $2 \le j \le k+2$. Note that $\lfloor\frac{2n+3k-1}{3k+2}\rfloor=2x+1$. Let $S_k=\{u_0, u_{2k+2}\} \cup (\cup_{i=1}^{x-1}\{u_{(3k+2)i+1}, u_{(3k+2)i+2k+2}\}) \cup \{u_{(3k+2)x+1}\}$. Since $S_k$ is a distance-$k$ resolving set of $C_n$ with $|S_k|=2x+1$, $\dim_k(C_n)\le |S_k|=2x+1=\lfloor\frac{2n+3k-1}{3k+2}\rfloor$.

Case 3: $n=(3k+2)x+j$, where $x \ge 1$ and $\lceil\frac{3k+5}{2}\rceil \le j \le 3k+1$. Note that $\lfloor\frac{2n+3k-1}{3k+2}\rfloor=2x+2$. Let $S_k=(\cup_{i=0}^{x-1}\{u_{(3k+2)i}, u_{(3k+2)i+2k+1}\}) \cup \{u_{(3k+2)x}, u_{\alpha}\}$, where $\alpha=\min\{n-1, (3k+2)x+2k+1\}$. Then $S_k$ is a distance-$k$ resolving set of $C_n$ with $|S_k|=2x+2$. So, $\dim_k(C_n)\le |S_k|=2x+2=\lfloor\frac{2n+3k-1}{3k+2}\rfloor$.
~\hfill
\end{proof}

\begin{remark}
Note that, for $n\ge 4$, Proposition~\ref{adim_c} is an immediate corollary of Theorem~\ref{kdim_cycle} when $k=1$.
\end{remark}

Next, we determine the distance-$k$ dimension of paths. We recall the following result.

\begin{proposition}\emph{\cite{adim}}\label{adim_p}
For $n \ge 4$, $\dim_1(P_n)=\lfloor \frac{2n+2}{5}\rfloor$.
\end{proposition}

Let $P_n$ be an $n$-path given by $u_0, u_1, \ldots, u_{n-1}$, where $n \ge 4$. Similar to the case for $C_n$, we define gaps and neighboring gaps of a vertex subset $M$ of $P_n$ analogously, where $|M| \ge 2$. If $d(u_0, M)=x$, then the set $\{u_0, u_1, \ldots, u_{x-1}\}$ is called the \emph{initial gap} of $M$; similarly, if $d(u_{n-1}, M)=y$, then the set $\{u_{n-1}, u_{n-2}, \ldots, u_{n-y}\}$ is called the \emph{terminal gap} of $M$. The union of the initial gap and the terminal gap of $M$ is called the \emph{union gap} of $M$. If $u_0\in M$ ($u_{n-1}\in M$, respectively), then the initial gap (terminal gap, respectively) is empty. The following lemma is analogous to Lemma~\ref{lem_cycle}, after adjusting for paths.

\begin{lemma}\label{lem_path}
For a positive integer $k$, let $M_k$ be a minimum distance-$k$ resolving set of $P_n$ for $n \ge k+3$. Then
\begin{itemize}
\item[(a)] Every gap of $M_k$ contains at most $2k+1$ vertices, the initial gap of $M_k$ contains at most $k+1$ vertices, and the terminal gap of $M_k$ contains at most $k+1$ vertices. Moreover, at most one gap of $M_k$ contains $2k+1$ vertices and the union gap of $M_k$ contains at most $2k+1$ vertices, but not both.
\item[(b)] If a gap of $M_k$ contains at least $k+1$ vertices, then any neighboring gaps contain at most $k$ vertices. If the initial gap or the terminal gap of $M_k$ contains at least one vertex, then its neighboring gap contains at most $k$ vertices.
\end{itemize}
\end{lemma}

\begin{proof}
Let $P_n$ be given by $u_0, u_1, \ldots, u_{n-1}$. 

(a) If there is a gap of $M_k$ containing $2k+2$ consecutive vertices $u_j, u_{j+1}, \ldots, u_{j+2k+1}$ of $P_n$, where $ 1 \le j \le n-2k-2$, then $r_{k, M_k}(u_{j+k})=r_{k, M_k}(u_{j+k+1})$. If the initial gap or the terminal gap of $M_k$, say the former without loss of generality, contains $k+2$ consecutive vertices $u_0, u_1, \ldots, u_{k+1}$, then $r_{k, M_k}(u_0)=r_{k, M_k}(u_1)$. If there exist two distinct gaps $u_p, u_{p+1}, \ldots, u_{p+2k}$ and $u_q, u_{q+1}, \ldots, u_{q+2k}$ of $M_k$, then $r_{k, M_k}(u_{p+k})=r_{k, M_k}(u_{q+k})$. If there exists a gap of $M_k$ containing $2k+1$ consecutive vertices, say $u_r, u_{r+1}, \ldots, u_{r+2k}$, and the union gap of $M_k$ containing $2k+1$ vertices, say $(\cup_{i=0}^{k}\{u_i\})\cup (\cup_{j=1}^{k}\{u_{n-j}\})$, then $r_{k, M_k}(u_{r+k})=r_{k, M_k}(u_{0})$.

(b) If a gap of $M_k$ contains at least $k+1$ vertices and one of its neighboring gap contains more than $k$ vertices, then there exist $2k+3$ consecutive vertices $u_j, u_{j+1}, \ldots, u_{j+2k+2}$ of $P_n$ such that $M_k \cap \{u_j, u_{j+1}, \ldots, u_{j+2k+2}\}=\{u_{j+k+1}\}$ and $r_{k, M_k}(u_{j+k})=r_{k, M_k}(u_{j+k+2})$. If the initial gap or the terminal gap of $M_k$, say the former, contains at least one vertex and its neighboring gap contains more than $k$ vertices, then there exist $k+3$ consecutive vertices $u_0, u_{1}, \ldots, u_{k+2}$ of $P_n$ such that $M_k \cap \{u_0, u_{1}, \ldots, u_{k+2}\}=\{u_1\}$ and $r_{k, M_k}(u_{0})=r_{k, M_k}(u_{2})$.~\hfill
\end{proof}

\begin{theorem}\label{kdim_path}
Let $n \ge 2$ and let $k$ be any positive integer.
\begin{itemize}
\item[(a)] If $n \le k+2$, then $\dim_k(P_n)=1$.
\item[(b)] If $k+3 \le n \le 3k+3$, then $\dim_k(P_n)=2$.
\item[(c)] If $n \ge 3k+4$, then
\begin{equation*}
\dim_k(P_n)=\left\{
\begin{array}{ll}
\lfloor\frac{2n+3k-1}{3k+2}\rfloor & \mbox{ if } n \equiv 0,1,\ldots, k+2 \pmod{(3k+2)},\\ 
\lfloor\frac{2n+4k-1}{3k+2}\rfloor & \mbox{ if } n \equiv k+3,\ldots,  \lceil \frac{3k+5}{2}\rceil-1 \pmod{(3k+2)},\\ 
\lfloor\frac{2n+3k-1}{3k+2}\rfloor & \mbox{ if }  n \equiv \lceil \frac{3k+5}{2}\rceil,\ldots, 3k+1 \pmod{(3k+2)}.
\end{array}\right.
\end{equation*}
\end{itemize}
\end{theorem}

\begin{proof}
Let $n \ge 2$ and let $k$ be a positive integer. 

(a) If $n\le k+2$, then $\dim_k(P_n)=1$ by Corollary~\ref{kdim_characterization}(a).

(b) Suppose $k+3 \le n \le 3k+3$. If $P_n$ is given by $u_0, u_1, \ldots, u_{n-1}$, then $\{u_k,u_{\alpha}\}$, where $\alpha=\min\{2k+1, n-1\}$, forms a distance-$k$ resolving set of $P_n$; thus $\dim_k(P_n) \le 2$. By Corollary~\ref{kdim_characterization}(a), $\dim_k(P_n)=2$.

(c) Let $n \ge 3k+4$. First, we show that  $\dim_k(P_n) \le \lfloor\frac{2n+3k-1}{3k+2}\rfloor$ if $n \equiv 0,1,\ldots, k+2 \pmod{(3k+2)}$ or $n \equiv \lceil \frac{3k+5}{2}\rceil,\ldots, 3k+1 \pmod{(3k+2)}$, and $\dim_k(P_n) \le \lfloor\frac{2n+4k-1}{3k+2}\rfloor$ if $n \equiv k+3,\ldots,  \lceil \frac{3k+5}{2}\rceil-1 \pmod{(3k+2)}$. By Lemmas~\ref{lem_path} and \ref{lem_cycle}, for $n \ge 3k+4$ every distance-$k$ resolving set of $P_n$ and $C_n$, respectively, has cardinality at least three. Moreover, there exists a minimum distance-$k$ resolving set $S$ of $C_n=P_n+e$ with a gap containing $2k+1$ vertices $u_{j}, u_{j+1} \ldots, u_{j+2k}$ in $C_n$, where $0 \le j \le n-1$ and the subscript is taken modulo $n$. If $e=u_{j+k}u_{j+k+1}$, then $S$ forms a distance-$k$ resolving set of $P_n$; thus $\dim_k(P_n) \le \dim_k(C_n)$ and the desired upper bounds follow from Theorem~\ref{kdim_cycle}. 

Second, we show that $\dim_k(P_n) \ge \lfloor\frac{2n+3k-1}{3k+2}\rfloor$; moreover, we show that $\dim_k(P_n) \ge \lfloor\frac{2n+4k-1}{3k+2}\rfloor$ if $n \equiv k+3,\ldots,  \lceil \frac{3k+5}{2}\rceil-1 \pmod{(3k+2)}$. Let $S_k$ be a minimum distance-$k$ resolving set of $P_n$ such that the union gap of $S_k$ contains $2k+1$ vertices. If $|S_k|=2 \ell$ for some positive integer $\ell$, then at most $\ell-1$ gaps contain more than $k$ vertices by Lemma~\ref{lem_path}(b) and those $\ell-1$ gaps contain at most $2k$ vertices by Lemma~\ref{lem_path}(a); thus, the number of vertices belonging to the gaps of $S_k$ or the union gap of $S_k$ is at most $2k(\ell-1)+k\ell+ (2k+1)=3k\ell+1$, and hence $n-2\ell \le 3k\ell+1$, which implies $|S_k|=2\ell \ge \dim_k(C_n)$. If $|S_k|=2\ell+1$ for some positive integer $\ell$, then at most $\ell-1$ gaps contain more than $k$ vertices by Lemma~\ref{lem_path}(b) and those $\ell-1$ gaps contain at most $2k$ vertices by Lemma~\ref{lem_path}(a); thus, the number of vertices belonging to the gaps of $S_k$ or the union gap of $S_k$ is at most $2k(\ell-1)+k(\ell+1)+(2k+1)=3k\ell+k+1$, and hence $n-(2\ell+1) \le 3k\ell+k+1$, which implies $|S_k|=2\ell +1 \ge \dim_k(C_n)$. In each case, $\dim_k(P_n) \ge \dim_k(C_n)$, and thus the desired lower bounds follow from Theorem~\ref{kdim_cycle}.~\hfill
\end{proof}

\begin{remark}
Note that, for $n\ge 4$, Proposition~\ref{adim_p} is an immediate corollary of Theorem~\ref{kdim_path} when $k=1$.
\end{remark}

Next, we use Theorem~\ref{kdim_path} to strengthen the bound in Theorem \ref{upper_diam}.

\begin{corollary}\label{upper_diam2}
If $G$ is a connected graph of order $n \ge 2$ and diameter $d \ge 1$, then $\dim_k(G) \le n - (d+1 - \lfloor\frac{2(d+1)+4k-1}{3k+2}\rfloor)$ for all $k \ge 1$.
\end{corollary}

\begin{proof}
Suppose that $u$ and $v$ are vertices in $G$ at distance $d$, and let $u = v_0, v_1, \ldots, v_d = v$ be a path of order $d+1$ with endpoints $u$ and $v$. If $X = \left\{v_0, \ldots, v_d\right\}$ and $H$ is the subgraph of $G$ restricted to $X$, then $H$ is a copy of $P_{d+1}$, so $H$ has a distance-$k$ resolving set $S$ of size at most $ \lfloor\frac{2(d+1)+4k-1}{3k+2}\rfloor$ by Theorem~\ref{kdim_path}. Then $(V(G) - X) \cup S$ is a distance-$k$ resolving set for $G$, which implies the corollary.
\end{proof}


\section{The effect of vertex or edge deletion on the distance-$k$ dimension of graphs}\label{s:delete}

Let $v$ and $e$, respectively, denote a vertex and an edge of a connected graph $G$ such that both $G-v$ and $G-e$ are connected graphs. First, we consider the effect of vertex deletion on distance-$k$ dimension of graphs. We recall the following results on the effect of vertex deletion on metric dimension and distance-$1$ dimension.

\begin{proposition}
\begin{itemize}
\item[(a)] \emph{\cite{wheel1}} $\dim(G)-\dim(G-v)$ can be arbitrarily large; 
\item[(b)] \emph{\cite{joc}} $\dim(G-v)-\dim(G)$ can be arbitrarily large.
\end{itemize}
\end{proposition}

\begin{proposition}\emph{\cite{broadcast}}\label{adim_vdel}
\begin{itemize}
\item[(a)] For any graph $G$, $\dim_1(G) \le \dim_1(G-v)+1$, where the bound is sharp.
\item[(b)] The value of $\dim_1(G-v)-\dim_1(G)$ can be arbitrarily large, as $G$ varies (see Figure~\ref{fig_vdeletion}).
\end{itemize}
\end{proposition}

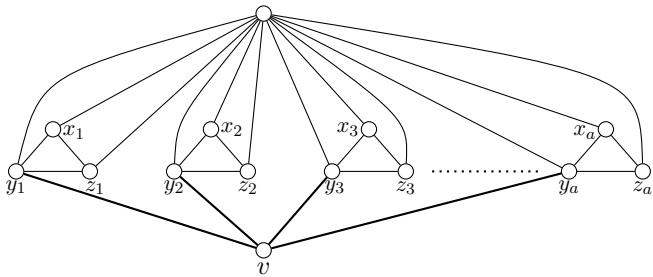
\begin{figure}[ht]
\centering
\begin{tikzpicture}[scale=.7, transform shape]

\node [draw, shape=circle, scale=.8] (0) at  (0, 3) {};
\node [draw, shape=circle, scale=.8] (1) at  (-4.7, 0) {};
\node [draw, shape=circle, scale=.8] (2) at  (-4, 0.8) {};
\node [draw, shape=circle, scale=.8] (3) at  (-3.3, 0) {};
\node [draw, shape=circle, scale=.8] (4) at  (-1.7, 0) {};
\node [draw, shape=circle, scale=.8] (5) at  (-1, 0.8) {};
\node [draw, shape=circle, scale=.8] (6) at  (-0.3, 0) {};
\node [draw, shape=circle, scale=.8] (7) at  (1.3, 0) {};
\node [draw, shape=circle, scale=.8] (8) at  (2, 0.8) {};
\node [draw, shape=circle, scale=.8] (9) at  (2.7, 0) {};
\node [draw, shape=circle, scale=.8] (10) at  (5.8, 0) {};
\node [draw, shape=circle, scale=.8] (11) at  (6.5, 0.8) {};
\node [draw, shape=circle, scale=.8] (12) at  (7.2, 0) {};
\node [draw, shape=circle, scale=.8] (v) at  (0, -1.5) {};

\node [scale=1.3] at (0,-1.85) {$v$};
\node [scale=1.1] at (-3.6,0.75) {$x_1$};
\node [scale=1.1] at (-4.7,-0.3) {$y_1$};
\node [scale=1.1] at (-3.2,-0.31) {$z_1$};
\node [scale=1.1] at (-0.6,0.8) {$x_2$};
\node [scale=1.1] at (-1.7,-0.3) {$y_2$};
\node [scale=1.1] at (-0.3,-0.31) {$z_2$};
\node [scale=1.1] at (1.6,0.8) {$x_3$};
\node [scale=1.1] at (1.35,-0.3) {$y_3$};
\node [scale=1.1] at (2.7,-0.31) {$z_3$};
\node [scale=1.1] at (6.1,0.75) {$x_a$};
\node [scale=1.1] at (5.8,-0.3) {$y_a$};
\node [scale=1.1] at (7.2,-0.31) {$z_a$};

\draw[thick,dotted] (3.2,0)--(5.3,0);

\draw(1)--(2)--(3)--(1);\draw(4)--(5)--(6)--(4);\draw(7)--(8)--(9)--(7);\draw(10)--(11)--(12)--(10);\draw(2)--(0);\draw(5)--(0);\draw(8)--(0);\draw(11)--(0);\draw(6)--(0);\draw(7)--(0);\draw(3)--(0);\draw(10)--(0);
\draw(1) .. controls(-4.25,1.85) .. (0);\draw(4) .. controls(-1.65,1) .. (0);\draw(9) .. controls(2.75,1) .. (0);\draw(12) .. controls(7.15,1.85) .. (0);

\draw[thick](v)--(1);\draw[thick](v)--(4);\draw[thick](v)--(7);\draw[thick](v)--(10);

\end{tikzpicture}
\caption{\small \cite{broadcast} Graphs $G$ such that $\dim(G-v) - \dim(G)=\dim_1(G-v)-\dim_1(G)$ can be arbitrarily large, where $a \ge 2$.}\label{fig_vdeletion}
\end{figure}

For graphs $G$ in Figure~\ref{fig_vdeletion}, note that $\diam(G)=\diam(G-v)=2$, where $a \ge2$. It was shown in~\cite{broadcast} that $\dim(G)=a+1$ and $\dim(G-v)=2a$. By Observation~\ref{obs_diam}(a), for any positive integer $k$, we have $\dim_k(G)=\dim(G)=a+1$ and $\dim_k(G-v)=\dim(G-v)=2a$, which implies the following

\begin{corollary}
Let $k$ be any positive integer. The value of $\dim_k(G-v)-\dim_k(G)$ can be arbitrarily large, as $G$ varies.
\end{corollary}

In contrast to the case for distance-$1$ dimension (see Proposition~\ref{adim_vdel}(a)), we show that $\dim_k(G)-\dim_k(G-v)$ can be arbitrarily large for $k \ge 2$.

\begin{proposition}
For any positive integer $k \ge 2$, $\dim_k(G)-\dim_k(G-v)$ can be arbitrarily large.
\end{proposition}

\begin{proof}
Let $k \ge 2$ and $x \ge 1$ be integers. Let $G=C_{5(3k+2)x}+K_1$ with the vertex $v$ in the $K_1$. Then $\dim_k(G)=\lfloor\frac{10(3k+2)x+2}{5}\rfloor=2(3k+2)x=6kx+4x$ by Corollary~\ref{cor_wheel1}, and $\dim_k(G-v)=\lfloor\frac{10(3k+2)x+3k-1}{3k+2}\rfloor=10x$ by Theorem~\ref{kdim_cycle}(b). So, $\dim_k(G)-\dim_k(G-v)=6kx+4x-10x=6(k-1)x \rightarrow \infty$ as $x \rightarrow \infty$ for $k \ge 2$.~\hfill 
\end{proof}

Next, we consider the effect of edge deletion on distance-$k$ dimension of graphs. Throughout the section, let $d_{H, k}(v_1, v_2)$ denote $d_k(v_1, v_2)$ in a graph $H$. We recall the following results on the effect of edge deletion on metric dimension and distance-$1$ dimension.

\begin{theorem}\emph{\cite{joc}}\label{dim_edge_deletion}
\begin{itemize}
\item[(a)] For any graph $G$ and any edge $e \in E(G)$, $\dim(G-e) \le \dim(G)+2$. 
\item[(b)] The value of $\dim(G) - \dim(G - e)$ can be arbitrarily large. 
\end{itemize}
\end{theorem}

\begin{theorem}\emph{\cite{broadcast}}\label{adim_edge_deletion}
For any graph $G$ and any edge $e \in E(G)$,
$\dim_1(G)-1 \le \dim_1(G-e) \le \dim_1(G)+1.$
\end{theorem}

The proof for Theorem~\ref{dim_edge_deletion}(a) in~\cite{joc}, adjusted for the case of distance-$k$ dimension, provides the following result. We include its proof to be self-contained.

\begin{proposition}\label{kdim_edge_deletion}
Let $k \ge 3$ be any integer. For any graph $G$ and any edge $e\in E(G)$, $\dim_k(G-e) \le \dim_k(G)+2$.
\end{proposition}

\begin{proof}
Let $S$ be a minimum distance-$k$ resolving set for $G$, and let $e=uw$. We show that $S\cup\{u,w\}$ is a distance-$k$ resolving set for $G-e$. Let $x$ and $y$ be distinct vertices in $V(G-e)=V(G)$ such that, for some $z\in S$, $d_{G, k}(x,z) \neq d_{G, k}(y,z)$ and $d_{G-e,k}(x,z)=d_{G-e,k}(y,z)$. We consider two cases.

\emph{Case 1: $d_{G,k}(x,z)=d_{G-e,k}(x,z)$ or $d_{G,k}(y,z)=d_{G-e,k}(y,z)$, but not both.} Suppose $d_{G,k}(y,z)=d_{G-e,k}(y,z)$. Then $d_{G,k}(y,z)=d_{G-e,k}(y,z)=d_{G-e,k}(x,z)>d_{G,k}(x,z)$, $d_{G, k}(x,z) \le k$, and the edge $e$ must lie on every $x-z$ geodesic in $G$. So, up to transposing the labels $u$ and $w$, we have $d_{G,k}(x,u)+d_{G,k}(u,w)+d_{G,k}(w,z)=d_{G,k}(x,z)$. Notice that $d_{G,k}(x,u)=d_{G-e,k}(x,u)$ since there is an $x-u$ geodesic in $G$ that does not use the edge $e$. Since $d_{G,k}(x,u)+d_{G,k}(u,z)=d_{G,k}(x,z)<d_{G,k}(y,z) \le d_{G,k}(y,u)+d_{G,k}(u,z)$, we must have $d_{G,k}(x,u)<d_{G,k}(y,u)$. Then $d_{G-e,k}(x,u)=d_{G,k}(x,u)<d_{G,k}(y,u)\le d_{G-e,k}(y,u)$ and $d_{G-e,k}(x,u) \le k-1$.

\emph{Case 2: $d_{G,k}(x,z) \neq d_{G-e,k}(x,z)$ and $d_{G,k}(y,z) \neq d_{G-e,k}(y,z)$.} In this case, the edge $e$ must lie on every $x-z$ geodesic and on every $y-z$ geodesic in $G$. Moreover, we must have either $d_{G,k}(x,z) <d_{G,k}(y,z) \le k$ or $d_{G,k}(y,z)<d_{G,k}(x,z) \le k$. Notice that if a geodesic from some vertex $a$ to another vertex $c$ traverses the edge $e$ in the order $u,w$ (as apposed to $w,u$), then a geodesic containing $e$ from any vertex $b$ to $c$ must also traverse $e$ in the order $u,w$. Suppose that $u$ is traversed before $w$ by an $x-z$ geodesic  and a $y-z$ geodesic (directed towards $z$) in $G$. Then an $x-u$ geodesic and a $y-u$ geodesic, neither containing the edge $e$, are obtained by removing a $u-z$ geodesic in $G$ from the $x-z$ geodesic and $y-z$ geodesic respectively. Thus, $d_{G-e,k}(x,u) \neq d_{G-e,k}(y,u)$.~\hfill
\end{proof}

\begin{remark}\label{kdim_edge_deletion12}
Note that Proposition~\ref{kdim_edge_deletion} and its proof hold when $k=1$ or $k=2$. For $k\in\{1,2\}$, we obtain the stronger result that $\dim_k(G-e) \le \dim_k(G)+1$. For $k = 1$ this follows from Theorem~\ref{adim_edge_deletion}. To see why it is true for $k = 2$, let $S$ be a minimum distance-$2$ resolving set of $G$, let $e=uw$, and let $x$ and $y$ be distinct vertices in $V(G-e)=V(G)$ such that $d_{G, 2}(x,z) \neq d_{G, 2}(y,z)$ and $d_{G-e,2}(x,z)=d_{G-e,2}(y,z)$ for some $z\in S$; further, suppose that the edge $e$ lies on every $x-z$ geodesic in $G$ and $d_{G,2}(x,u) < d_{G,2}(x,w)$.

First, we consider Case 1. Then $0<d_{G,2}(x,z) \le 2$; notice that $x\neq z$ since the edge $e$ lies on every $x-z$ geodesic in G. If $d_{G,2}(x,z)=1$, then $e=uw=xz$; if $d_{G,2}(x,z)=2$, then $x=u$ or $xu\in E(G)$. In each case, $S\cup\{u\}$ forms a distance-$2$ resolving set for $G-e$. Next, we consider Case 2. Then $d_{G,2} (x,z) < d_{G,2} (y,z) \le 2$ or $d_{G,2}(y,z)<d_{G,2}(x,z)\le 2$, say the former; then $e=uw=xz$ and $x$ lies on every $y-z$ geodesic in $G$. So $S \cup \{u\}$ forms a distance-$2$ resolving set for $G-e$. Therefore, $\dim_2(G-e) \le \dim_2(G)+1$. For graphs $G$ satisfying $\dim_2(G-e) = \dim_2(G)+1$, see Figure~\ref{fig_kdim_edge_old}, where $a,b,c\ge2$; one can easily check that $R=(\cup_{i=1}^{a-1}\{x_i\})\cup(\cup_{i=1}^{b-1}\{y_i\})\cup(\cup_{i=1}^{c-1}\{z_i\})$ forms a minimum distance-$2$ resolving set of $G-e$ with $|R|=a+b+c-3$ and that $R'=R-\{z_1\}$ forms a minimum distance-$2$ resolving set for $G$ with $|R'|=a+b+c-4$.
\end{remark}

\begin{remark}
The bound in Proposition~\ref{kdim_edge_deletion} is sharp. For any integer $k \ge 3$, let $G$ be the graph in Figure~\ref{fig_kdim_edge_old} and let $e=x_1z_1$. Let $L_1=\cup_{i=1}^{a}\{x_i\}$, $L_2=\cup_{i=1}^{b}\{y_i\}$ and $L_3=\cup_{i=1}^{c}\{z_i\}$, where $a,c \ge 3$ and $b\ge 2$.

First, we show that $\dim_k(G-e)=a+b+c-3$. Note that any two vertices in $L_i$ are twin vertices in $G-e$, where $i\in\{1,2,3\}$. So, for any distance-$k$ resolving set $S$ of $G-e$, we have $|S \cap L_1| \ge a-1$, $|S \cap L_2| \ge b-1$, and $|S \cap L_3| \ge c-1$ by Observation~\ref{obs_twin}(b); thus, $\dim_k(G-e) \ge a+b+c-3$. On the other hand, $(L_1 \cup L_2 \cup L_3)-\{x_1, y_1,z_1\}$ forms a distance-$k$ resolving set of $G-e$, and hence $\dim_k(G-e) \le a+b+c-3$. Thus, $\dim_k(G-e)=a+b+c-3$.

Second, we show that $\dim_k(G)=a+b+c-5$. For any distance-$k$ resolving set $S'$ of $G$, we have $|S' \cap (L_1-\{x_1\})| \ge a-2$, $|S' \cap L_2| \ge b-1$, and $|S' \cap (L_3-\{z_1\})| \ge c-2$ by Observation~\ref{obs_twin}(b); thus, $\dim_k(G) \ge a+b+c-5$. Since $(L_1\cup L_2\cup L_3)-\{x_1, x_2, y_1, z_1, z_2\}$ forms a distance-$k$ resolving set of $G$, $\dim_k(G) \le a+b+c-5$. So, $\dim_k(G)=a+b+c-5$.

Therefore, $\dim_k(G-e)=\dim_k(G)+2$ for $k \ge 3$.
\end{remark}

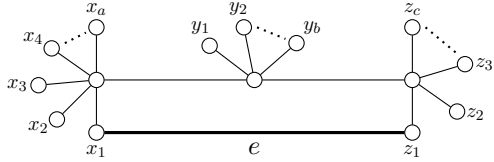
\begin{figure}[ht]
\centering
\begin{tikzpicture}[scale=.7, transform shape]

\node [draw, shape=circle, scale=.8] (0) at  (0, 0) {};
\node [draw, shape=circle, scale=.8] (1) at  (3, 0) {};
\node [draw, shape=circle, scale=.8] (2) at  (6, 0) {};

\node [draw, shape=circle, scale=.8] (01) at  (0, 1) {};
\node [draw, shape=circle, scale=.8] (02) at  (-0.85, 0.6) {};
\node [draw, shape=circle, scale=.8] (03) at  (-1.1, -0.1) {};
\node [draw, shape=circle, scale=.8] (04) at  (-0.75, -0.75) {};
\node [draw, shape=circle, scale=.8] (05) at  (0, -1) {};

\node [draw, shape=circle, scale=.8] (11) at  (2.15, 0.65) {};
\node [draw, shape=circle, scale=.8] (12) at  (2.8, 1) {};
\node [draw, shape=circle, scale=.8] (13) at  (3.8, 0.7) {};

\node [draw, shape=circle, scale=.8] (21) at  (6, 1) {};
\node [draw, shape=circle, scale=.8] (22) at  (7, 0.3) {};
\node [draw, shape=circle, scale=.8] (23) at  (6.85, -0.6) {};
\node [draw, shape=circle, scale=.8] (24) at  (6, -1) {};

\node [scale=1] at (0,1.35) {$x_a$};
\node [scale=1] at (-1.2,0.8) {$x_4$};
\node [scale=1] at (-1.5,-0.1) {$x_3$};
\node [scale=1] at (-1.1,-0.9) {$x_2$};
\node [scale=1] at (0,-1.35) {$x_1$};
\node [scale=1] at (4.1,1) {$y_b$};
\node [scale=1] at (2.7,1.35) {$y_2$};
\node [scale=1] at (1.98,1) {$y_1$};
\node [scale=1] at (6,1.35) {$z_c$};
\node [scale=1] at (7.37,0.3) {$z_3$};
\node [scale=1] at (7.2,-0.65) {$z_2$};
\node [scale=1] at (6,-1.35) {$z_1$};

\node [scale=1.3] at (3,-1.3) {$e$};

\draw(05)--(0)--(1)--(2)--(23);\draw(11)--(1)--(12);\draw(21)--(2)--(22);\draw(01)--(0);\draw(02)--(0)--(03);\draw[very thick](05)--(24);\draw(0)--(04);\draw(1)--(13);\draw(2)--(24);
\draw[thick, dotted] (-0.6,0.7)--(-0.2,0.9);\draw[thick, dotted] (3.05,1)--(3.59,0.78);\draw[thick, dotted] (6.25,1)--(6.85,0.5);

\end{tikzpicture}
\caption{\small Graphs $G$ with $\dim_2(G-e)=\dim_2(G)+1$ and $\dim_k(G-e)=\dim_k(G)+2$ for $k \ge 3$.}\label{fig_kdim_edge_old}
\end{figure}

In contrast to Theorem~\ref{adim_edge_deletion}, we show that $\dim_k(G)-\dim_k(G-e)$ can be arbitrarily large for any integer $k \ge 2$.

\begin{theorem}
For any integer $k\ge 2$, the value of $\dim_k(G)-\dim_k(G-e)$ can be arbitrarily large.
\end{theorem}

\begin{proof}
Let $G$ be the graph in Figure~\ref{fig_kdim_edge}. For each $i\in\{1,2,\ldots, a\}$, $N_G(x_i)=N_G(y_i)=\{z_i, z'_i\}=N_{G-e}(x_i)=N_{G-e}(y_i)$. Let $k \ge 2$ and $a\ge 2$ be any integers. Let $S$ be any minimum distance-$k$ resolving set for $G-e$, and let $S'$ be any distance-$k$ resolving set for $G$. By Observation~\ref{obs_twin}(b), $S \cap\{x_i, y_i\} \neq\emptyset$ and $S' \cap\{x_i, y_i\} \neq\emptyset$ for each $i \in \{1,2,\ldots a\}$; without loss of generality let $S_0=\cup_{i=1}^{a}\{x_i\} \subseteq S \cap S'$.

First, we show that $\dim_k(G-e)=a+1$. Since $r_{k,S_0}(z_i)=r_{k,S_0}(z'_i)$ for each $i\in\{1,2,\ldots, a\}$ in $G-e$, $|S| \ge a+1$, and hence $\dim_k(G-e) \ge a+1$. Since $S_0 \cup \{v\}$ forms a distance-$k$ resolving set of $G-e$, $\dim_k(G-e) \le a+1$. So, $\dim_k(G-e)=a+1$.

Second, we show that $\dim_k(G)=2a$. Note that, for each $i \in \{1,2,\ldots, a\}$,  $r_{k,S_0}(z_i)=r_{k,S_0}(z'_i)$ in $G$ and $R_k\{z_i, z'_i\}=\{z_i, z'_i, t_i\}$; thus, $S' \cap \{z_i, z'_i, t_i\} \neq \emptyset$ for each $i\in\{1,2,\ldots,a\}$. So, $|S'| \ge 2a$, and hence $\dim_k(G) \ge 2a$. Since $S_0\cup(\cup_{i=1}^{a} \{z_i\})$ forms a distance-$k$ resolving set of $G$, $\dim_k(G)\le 2a$. Thus, $\dim_k(G)=2a$.

Therefore, $\dim_k(G)-\dim_k(G-e)=2a-(a+1)=a-1 \rightarrow \infty$ as $a \rightarrow \infty$.~\hfill
\end{proof}

\begin{figure}[ht]
\centering
\begin{tikzpicture}[scale=.7, transform shape]

\node [draw, shape=circle, scale=.8] (1) at  (4, 1.5) {};
\node [draw, shape=circle, scale=.8] (2) at  (6, 1.5) {};
\node [draw, shape=circle, scale=.8] (3) at  (3, 0) {};
\node [draw, shape=circle, scale=.8] (4) at  (5.5, 0) {};
\node [draw, shape=circle, scale=.8] (5) at  (4,-1.5) {};
\node [draw, shape=circle, scale=.8] (6) at  (6, -1.5) {};
\node [draw, shape=circle, scale=.8] (7) at  (6, -3.5) {};
\node [draw, shape=circle, scale=.8] (8) at  (8, -3.5) {};

\node [draw, shape=circle, scale=.8] (a1) at  (1.5, 0) {};
\node [draw, shape=circle, scale=.8] (a2) at  (2.25, 0.5) {};
\node [draw, shape=circle, scale=.8] (a3) at  (2.25, -0.5) {};

\node [draw, shape=circle, scale=.8] (b1) at  (2.5, -1.5) {};
\node [draw, shape=circle, scale=.8] (b2) at  (3.25, -1) {};
\node [draw, shape=circle, scale=.8] (b3) at  (3.25, -2) {};

\node [draw, shape=circle, scale=.8] (c1) at  (4.5, -3.5) {};
\node [draw, shape=circle, scale=.8] (c2) at  (5.25, -3) {};
\node [draw, shape=circle, scale=.8] (c3) at  (5.25, -4) {};

\node [scale=1.1] at (2.4,0.8) {$x_1$};
\node [scale=1.1] at (2.35,-0.8) {$y_1$};
\node [scale=1.1] at (3.4,-0.7) {$x_2$};
\node [scale=1.1] at (3.4,-2.3) {$y_2$};
\node [scale=1.1] at (5.3,-2.7) {$x_a$};
\node [scale=1.1] at (5.3,-4.3) {$y_a$};
\node [scale=1.1] at (5.6,-0.38) {$t_1$};
\node [scale=1.1] at (6.15,-1.85) {$t_2$};
\node [scale=1.1] at (8.15,-3.85) {$t_a$};
\node [scale=1.2] at (6.33,1.5) {$v$};
\node [scale=1.1] at (1.15,0) {$z_1$};
\node [scale=1.1] at (2.15,-1.5) {$z_2$};
\node [scale=1.1] at (4.15,-3.5) {$z_a$};

\node [scale=1.3] at (5,1.75) {\bf$e$};

\draw(1)--(3)--(4)--(2);\draw(1)--(5)--(6)--(2);\draw(1)--(7)--(8)--(2);
\draw(a1)--(a2)--(3)--(a3)--(a1);
\draw(a1) .. controls (2.25,1.15) .. (1);
\draw(b1)--(b2)--(5)--(b3)--(b1);\draw(b1)--(1);
\draw(c1)--(c2)--(7)--(c3)--(c1);\draw(c1)--(1);
\draw[very thick](1)--(2);

\draw[thick, dotted] (6.35,-2)--(7.8,-3.3);

\end{tikzpicture}
\caption{\small Graphs $G$ such that $\dim_k(G) - \dim_k(G-e)$ can be arbitrarily large, where $k\ge 2$ and $a \ge 2$.}\label{fig_kdim_edge}
\end{figure}
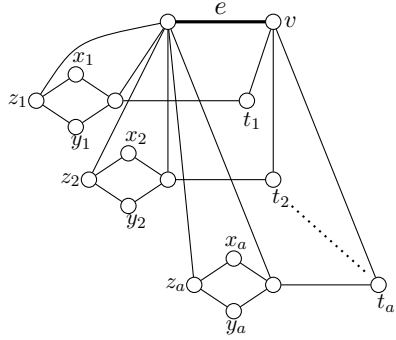

\end{document}